\documentclass[12pt,a4paper]{article}
\usepackage{amsfonts,amssymb,amsmath,graphicx}
\usepackage{theorem}
\usepackage{bm}
\theoremstyle{change}{\theorembodyfont{\rm}
  \newtheorem{defi}{Definition.}[section]
  \newtheorem{rem}[defi]{Remark.}
  }
  \newtheorem{lem}[defi]{Lemma.}
  \newtheorem{prop}[defi]{Proposition.}
  \newtheorem{thm}[defi]{Theorem.}
  \newtheorem{cor}[defi]{Corollary.}
\newenvironment{proof}
{\begin{trivlist} \item {\sl Proof:}}
{\/ $\square$ \end{trivlist}}
\newcommand{\BB}{{\mathbb B}}
\newcommand{\PP}{{\mathbb P}}
\newcommand{\cB}{{\mathcal B}}
\newcommand{\cC}{{\mathcal C}}
\newcommand{\cK}{{\mathcal K}}
\newcommand{\cP}{{\mathcal P}}
\newcommand{\bE}{{\bm{E}}}
\newcommand{\bH}{{\bm{H}}}
\newcommand{\bS}{{\bm{S}}}
\newcommand{\bU}{{\bm{U}}}

\DeclareMathOperator{\Fix}{Fix}
\DeclareMathOperator{\id}{id}

\newcommand{\eps}{{\varepsilon}}
\DeclareMathOperator{\GF}{GF}
\DeclareMathOperator{\GL}{GL}
\DeclareMathOperator{\PG}{PG}
\DeclareMathOperator{\M}{M}
\let\phi=\varphi

\let\theta=\vartheta
\newcommand{\tsk}{$t$-$(s,k,\lambda_t)$}

\newcommand{\SDelimArray}[4]{\hbox{\scriptsize\arraycolsep=.5\arraycolsep
  $\left#1\!\!\begin{array}{*{#3}{c}}#4\end{array}\!\!\right#2$}}

\newcommand{\SMat}{\SDelimArray()}

\date{}

\begin{document}
\title{Divisible designs from twisted dual numbers}
\author{Andrea Blunck \and Hans Havlicek\thanks{Corresponding author.} \and Corrado Zanella}

\maketitle

\centerline{\emph{Dedicated to Helmut M\"{a}urer on the occasion of his 70th birthday}}

\begin{abstract}\noindent
The generalized chain geometry over the local ring $K(\eps;\sigma)$ of twisted
dual numbers, where $K$ is a finite field, is interpreted as a divisible design
obtained from an imprimitive group action. Its combinatorial properties as well
as a geometric model in $4$-space are investigated.


\end{abstract}
\parskip .8mm
\parindent0cm
\section{Preliminaries}\label{sec:pre}

This paper deals with a special class of \emph{divisible designs}, namely,
those that are chain geometries over certain finite local rings, and their
representation in projective space.

A finite geometry $\Sigma=(\cP,\cB,\parallel)$, consisting of a set $\cP$ of
\emph{points}, a  set $\cB$  of \emph{blocks}, and an equivalence relation
$\parallel$ (\emph{parallel}) on $\cP$,  is called a \tsk-\emph{divisible
design} ($t$-DD for short), if there exist positive integers $t,s,k,\lambda_t$
such that the following axioms hold:
\begin{itemize}
\item Each  block $B$ is a   subset of $\cP$  containing $k$ pairwise
non-parallel points. \item Each parallel class consists of $s$ points. \item
For  each set $Y$ of $t$ pairwise non-parallel points there exist exactly
$\lambda_t$ blocks containing $Y$. \item $t\leq k \leq v/s$, where $v:=|\cP|$.
\end{itemize}

Note that sometimes DDs are called ``group divisible designs''.

A DD with trivial parallel relation, i.e.\ with $s=1$, is an ordinary
\emph{design}. A DD with $k=v/s$ is called \emph{transversal}. In the
subsequent sections we shall deal with transversal $3$-DDs.

A method to construct DDs with a large group of automorphisms is due to A.G.
Spera \cite{spera-92}, using imprimitive group actions: Let $G$ be a group
acting on a (finite) set $\cP$ of ``points'' and leaving invariant an
equivalence relation $\parallel $ (``parallel'').  Let $t$ be a positive
integer such that there are at least $t$ parallel classes, and let $B_0$ be a
set of $k\ge t$ pairwise non-parallel points (the ``base block''). Assume that
$G$ acts transitively on the set of $t$-tuples of pairwise non-parallel points.
Let $\cB$ be the orbit of $B_0$ under $G$, i.e.\ $\cB= \{B_0^g\mid g\in G\}$.
Then $\Sigma=(\cP, \cB,\parallel)$ is a $t$-DD with
\begin{equation}\label{eq:lambda}
\lambda_t= \frac{|G|}{|G_{B_0}|}\cdot\frac{{k\choose t}}{{v/s\choose t} s^t} 
\end{equation}
where $G_{B_0}$ is the (setwise) stabilizer of $B_0$ in $G$ (see
\cite[Prop.~2.3]{spera-92}).

The projective line $\PP(R)$ over a finite local ring $R$ is endowed with an
equivalence relation (usually denoted by $\parallel$). It is invariant under
the action of the general linear group $\GL_2(R)$ on $\PP(R)$. Since $\GL_2(R)$
acts transitively on the set of triples of non-parallel points, any $k$-set
($k\geq 3)$ of mutually non-parallel points of $\PP(R)$ can be chosen as base
block $B_0$ in order to apply Spera's construction of a DD. This is, of course,
a very general approach. Therefore, it is not surprising that not too much can
be said about the corresponding $3$-DDs. It is straightforward to express their
parameters $v$ and $s$, as well as the order of the group $\GL_2(R)$, in terms
of $|R|$ and $|I|$, i.e.\ the cardinality of the unique maximal ideal $I$ of
the given ring $R$. However, in order to calculate the parameter $\lambda_3$ by
virtue of (1), one needs to know the order of the stabilizer of $B_0$ in
$\GL_2(R)$. But it seems hopeless to calculate this order without further
information about the base block $B_0$.
\par
If $R$ is even a finite local algebra over a field $F$, say, then the
projective line $\PP(F)$ over $F$ can be considered as a subset of $\PP(R)$,
and it can be chosen as a base block. All $3$-DDs obtained in this way satisfy
$\lambda_3=1$; they are---up to notational differences---precisely the
(classical) \emph{chain geometries\/} $\Sigma(F,R)$; see \cite{benz-73},
\cite{blu+he-05} or \cite{herz-95}. This was pointed out by Spera
\cite[Example~2.5]{spera-92}. In the cited paper also a series of interesting
DDs are constructed from base blocks which are certain subsets of $\PP(F)$. See
also \cite{giese+h+s-05} for similar results. We mention in passing that
higher-dimensional projective spaces over local algebras give rise to $2$-DDs
\cite{spera-95}.
\par
The divisible designs which are constructed in the present paper arise also
from chain geometries. However, we use this term in a more general form which
was introduced in \cite{blu+h-00a} just a few years ago. The essential
difference is as follows: We consider a finite local ring $R$ containing a
subfield $K$ which is not necessarily in the centre of $R$. Thus \emph{$R$ need
not be a $K$-algebra}, but of course it is an algebra over some subfield of
$K$. As before, we can define $\PP(K)\subset\PP(R)$ to be the base block. This
gives a $3$-DD which coincides with the (generalized) chain geometry
$\Sigma(K,R)$. It is possible to express the parameter $\lambda_3$ of this DD
in algebraic terms (see \cite[Theorem~2.4]{blu+h-00a}), but this is not very
explicit in the general case. Therefore, we focus our attention on a particular
class of local rings, namely twisted dual numbers. If the ``twist'' is
non-trivial, then $3$-DDs with parameter $\lambda_3=|K|$ are obtained.
\par
In Section 4 we present an alternative description of our $3$-DDs in a finite
projective space over $K$.

\section{Twisted dual numbers}\label{sec:tdn}

Let $R$ be a (not necessarily commutative) local ring containing a (not
necessarily central) subfield $K$. In view of our objective to construct DDs,
we will later restrict ourselves to finite rings and fields, and hence we
assume from the beginning that $K$ is commutative. As usual, we denote by $R^*$
the group of units (invertible elements) of $R$. We set $I:=R\setminus R^*$;
since $R$ is local we have that $I$ is an ideal.

The ring $R$ is in a natural way a left vector space over $K$, sometimes
written as ${}_KR$. We assume that $\dim(_KR)=2$. Moreover, we assume that $R$
is not a field. We want to determine the structure of $R$. The ideal $I$ is a
non-trivial subspace of the vector space ${ }_KR$. So $\dim(_KI)=1$, and
$I=K\eps$ for some $\eps\in R\setminus K$. Then $1,\eps$ is a basis of ${
}_KR$, and we may write $R=K + K\eps$.

In order to describe the multiplication in $R$ we first observe that $\eps^2\in
I$, so $\eps^2= b\eps$ for some $b\in K$. This implies $(\eps-b)\eps=0$, whence
also $\eps-b\in I$ and so $b=0$. For each $x\in K$ we have $\eps x\in I$, so
there is a unique $x'\in K$ such that $\eps x=x'\eps$. One can easily check
that $\sigma:x\mapsto x'$ is an injective field endomorphism.

Conversely, given a field $K$ and an injective endomorphism $\sigma$ of $K$ we
obtain a ring of \emph{twisted dual numbers} $R=K(\eps;\sigma)=K + K\eps$ with
multiplication
\begin{equation*}
(a+b\eps)(c+d\eps)=ac+(ad+bc^\sigma)\eps.
\end{equation*}
In the special case that $\sigma=\id$ this is the well known commutative ring
$K(\eps)$ of \emph{dual numbers} over $K$.

The subfield  $\Fix(\sigma)$ of $K$ fixed elementwise by $\sigma$ will be
called $F$. So $F=K$ if, and only if, $\sigma=\id$.

The units of $R$ are exactly the elements of $R\setminus I=K^*  + K\eps$. One
can easily check that the inverse of a unit $u=a+b\eps$ (with $a,b\in K, a\ne
0$) is
\begin{equation}\label{eq:inv}
u^{-1}=a^{-1}-a^{-1}b(a^\sigma)^{-1}\eps.
\end{equation}

Later we shall need the following algebraic statements on $R=K(\eps;\sigma)$.
\begin{lem}\label{lem:semi}
The multiplicative group $R^*$ is the semi-direct product of $K^*$ and the
normal subgroup
\begin{equation}\label{eq:U}
U =1 + K\eps =\{1+b\eps\mid b\in K\}.
\end{equation}
\end{lem}
\begin{proof}
Direct computation, using (\ref{eq:inv}) for showing that $U$ is normal in
$R^*$.
\end{proof}

\begin{lem}\label{lem:N}
Let $N$ be the normalizer of $K^*$ in $R^*$, i.e.,
\begin{equation}\label{eq:N}
N=\{n\in R^*\mid n^{-1}K^*n=K^*\}.
\end{equation}
Then $N=R^*$ if $\sigma=\id$ and $N=K^*$ otherwise.
\end{lem}
\begin{proof}
For $\sigma=\id$ the assertion is clear. So let $\sigma\ne\id$ and
$n=a+b\eps\in N$. Take an element  $x\in K\setminus F$. Using (\ref{eq:inv}) we
get $n^{-1}xn=x+a^{-1}b(x-x^{\sigma})\eps$, which must belong to $K$ since
$n\in N$. Because of our choice of $x$ we have  $x-x^\sigma\ne 0$, whence
$b=0$, as desired.
\end{proof}

\section{The associated  DD}\label{sec:DD}
In this section we construct a $3$-DD using the ring $R=K(\eps;\sigma)$. The
construction is a special case of Spera's construction method described in
Section~\ref{sec:pre} (see also \cite[Section~2.3]{havl-04}). On the other
hand, the resulting DD is nothing else than the (generalized) \emph{chain
geometry} over $(K,R)$ (compare \cite{blu+h-00a}, for details on ordinary chain
geometries see \cite{blu+he-05}, \cite{herz-95}).

From now on we assume that $R$, and hence also $K$ and $F$,  are finite. Then
$F=\GF(m)$ for some prime power $m$, and $K=\GF(q)$ with $q$ a power of $m$.
Moreover, $\sigma$ now is an automorphism of $K$, namely, $\sigma: x\mapsto
x^m$.

The construction is based on the action of the group $G =\GL_2(R)$ of
invertible $2\times2$-matrices with entries in $R$ on the \emph{projective
line} over $R$, i.e., on the set
\begin{equation}\label{eq:PR}
    \PP(R)=\{R(a,b)\le R^2\mid \exists\, c,d\in R:\SMat2{a&b\\c&d}\in G\}.
\end{equation}
Since $R$ is local, each pair $(a,b)$ as in (\ref{eq:PR}) has the property that
at least one of the two elements $a,b$ is invertible, because otherwise the
existence of an inverse matrix $\SMat2{x&*\\y&*}$ would lead to the
contradiction $1=ax+by\in I$. So $\PP(R)$ is the disjoint union
\begin{equation}\label{eq:PRlocal}
\PP(R)=\{R(x,1)\mid x\in R\}\cup\{R(1,z)\mid z\in I\}.
\end{equation}
On $\cP =\PP(R)$ we have an equivalence relation $\parallel$ given by
\begin{equation}\label{eq:par}
R(a,b)\parallel R(c,d) :\iff \SMat2{a&b\\c&d}\notin G.
\end{equation}
More explicitly, this means for arbitrary $x,y\in R, z,w\in I$:
 \begin{equation}\label{eq:parlocal}
    R(1,z)\parallel R(1,w); \; R(x,1) \nparallel R(1,z); \; \big(R(x,1)\parallel
    R(y,1) \Leftrightarrow x-y\in I\big).
\end{equation}
Using the description in (\ref{eq:parlocal}) one can see that $\parallel$ in
fact is an equivalence relation.

Let us recall two facts (see \cite[1.2.2]{herz-95} and
\cite[Prop.~1.3.3]{herz-95}, where non-parallel points are called ``distant''):
The group $G$ acts on $\cP$ leaving $\parallel$ invariant. Moreover, $G$ acts
transitively on the set of triples of pairwise non-parallel points of $\cP$. By
virtue of this action of $G$ and (\ref{eq:parlocal}), any two parallel classes
have the same cardinality $s=|I|$.

In order to apply Spera's method we now need a \emph{base block} consisting of
pairwise non-parallel points. As usual for chain geometries, we use the
projective line over $K$.

Since $K$ is a subfield of $R$, the projective line $\PP(K)$ can be seen as a
subset $B_0$ of $\cP=\PP(R)$ as follows:

\begin{equation}\label{eq:B0}
B_0 =\PP(K)=\{R(x,1)\mid x\in K\}\cup\{R(1,0)\}.
\end{equation}

Let $\cB =B_0^G$. Then we get the following.

\begin{thm}\label{prop:DD}
The structure $\Sigma=(\cP,\cB,\parallel)$ is a transversal $3$-DD with
parameters $v=q^2+q$, $s=q$, $k=q+1 (=v/s)$ and
\begin{equation*}
    \lambda_3=
    \begin{cases}
    1 & \mbox{\;  if \; } \sigma= \id, \\
    q & \mbox{\;  if \; } \sigma \ne \id. \end{cases}
\end{equation*}
\end{thm}
\begin{proof}
From Spera's theorem we know that $\Sigma$ is a DD. The values of $v$, $s$, and
$k$ are obtained from (\ref{eq:PRlocal}), (\ref{eq:parlocal}), and
(\ref{eq:B0}), respectively. By \cite[Theorem~2.4]{blu+h-00a}, we have
$\lambda_3=|R^*|/|N|$, where $N$ is the normalizer defined in (\ref{eq:N}). By
Lemma \ref{lem:N} we have two cases: If $\sigma=\id$, the normalizer $N$
coincides with $R^*$ and so $\lambda_3=1$. If $\sigma\ne \id$, the normalizer
$N$ equals $K^*$, whence $\lambda_3=|R^*|/|N|=(q-1)q/(q-1)=q$.
\end{proof}
The equation $\sigma=\id$ holds precisely when $K$ lies in the centre of $R$;
in this case our DD is an ordinary chain geometry, namely, the \emph{Miquelian
Laguerre plane} over the algebra of dual numbers (see
\cite[I.2,~II.4]{benz-73}).

We mention here that the  parameter $\lambda_3$ could also be computed directly
using the formula
\begin{equation}\label{eq:lambdaA}
\lambda_3=\frac{|G|}{|G_{B_0}|}\cdot\frac{1}{s^3}
\end{equation}
(see (\ref{eq:lambda}),  note that $k=v/s$).

We add without proof that the $3$-DD $\Sigma$ can also be described as a
\emph{lifted DD\/} in the sense of \cite[Theorem~2.5]{blu+h+z-07}, using the
point set $\cP$, the equivalence relation $\parallel$, the group
 \begin{equation*}\label{eq:H}
H=\left\{\SMat2{1+a\eps&b\eps\\c\eps&1+d\eps}\mid a,b,c,d\in K\right\}
\end{equation*}
which acts on $\cal P$, and the base block $B_0$ as (trivial) base DD. However,
this alternative approach does not immediately show the large group of
automorphisms given by the action of $G$ on $\cP$.

We now have a closer look at the case $\sigma\ne\id$. We want to determine the
$q$ blocks through three given pairwise non-parallel points more explicitly.
Because of the transitivity properties of $G$ it suffices to consider the
points $\infty=R(1,0)$, $0=R(0,1)$, $1=R(1,1)$. From
\cite[Theorem~2.4]{blu+h-00a} we know the following: The blocks through
$\infty, 0,1$ are exactly the images of $B_0$ under the group
\begin{equation}\label{eq:Rdach}
    \widehat{R^*} =\left\{\SMat2{u&0\\0&u}\mid u\in R^*\right\},
\end{equation}
and two elements $\omega=\SMat2{u&0\\0&u}$ and $\omega'=\SMat2{u'&0\\0&u'}$ of
$\widehat{R^*}$ determine the same block if, and only if, $Nu=Nu'$, with $N$ as
in   (\ref{eq:N}). So from Lemmas \ref{lem:N} and \ref{lem:semi} we obtain:

\begin{lem}\label{lem:blocks}
Let $\sigma\ne\id$. Then the blocks containing $\infty=R(1,0)$, $0=R(0,1)$,
$1=R(1,1)$ are exactly the $q$ sets
\begin{equation}\label{eq:blocks}
    B_0^\omega,\; {\mbox{with}\; \; }\omega=\SMat2{1+b\eps &0\\0&1+b\eps }, \; b
\in K.
\end{equation}
\end{lem}

We now give an explicit description of  the action of the group
\begin{equation}\label{eq:Udach}
\widehat{U} =\left\{\SMat2{u&0\\0&u}\mid u\in U\right\}=\left\{\SMat2{1+b\eps
&0\\0&1+b\eps}\mid b\in K\right\},
\end{equation}
associated to $U$ (see (\ref{eq:U})), on $\cP=\PP(R)$.

A direct calculation shows that each $\omega=\SMat2{u&0\\0&u}$, with $u\in
R^*$, acts on $\cP$ via ``conjugation'' as follows:
\begin{equation}
\omega: R(x,1)\mapsto R(u^{-1}xu,1),\; \; R(1,z)\mapsto R(1,u^{-1}zu),
\end{equation}
where, as before, $x\in R, z\in I$. For $u=1+b\eps\in U$ this yields, using
(\ref{eq:inv}),
\begin{equation}\label{eq:omegaU}
\omega: R(x,1)\mapsto R(x+b(x_1-x_1^\sigma)\eps,1),\; \; R(1,z)\mapsto R(1,z),
\end{equation}
where $x=x_1+x_2\eps$. So the mapping $\omega\in\widehat U$ of
(\ref{eq:omegaU}) maps each point to a parallel one. Moreover, it fixes exactly
those elements of the base block $B_0=\PP(K)$ that belong to the subset
$\PP(F)$. This subset in turn is the intersection of all blocks through
$\infty, 0,1$ (compare (\ref{eq:blocks})); such intersections are also called
\emph{traces} (in German: ``F\"{a}hrten'', see \cite{benz-73}, \cite{blu+h-00a}).

We consider a parallel class on which $\widehat U$ does not act trivially. By
(\ref{eq:omegaU}) this is  the parallel class of some point $p=R(x_1,1)$, where
$x_1\in K\setminus F$ and consequently  $p\in B_0\setminus\PP(F)$. Then
$\widehat U $ acts regularly on the parallel class under consideration. As a
matter of fact, for  each $p'$ parallel to $p$, which  has the form
$p'=R(x_1+x_2\eps,1)$, there is a unique $b\in K$ with $x_2=b(x_1-x_1^\sigma)$,
so $p^\omega=p'$, with $\omega$ as in (\ref{eq:omegaU}). This means that for
each $p'\parallel p$ there is exactly one block through $\infty, 0,1$ that
contains $p'$ (and each block through $\infty, 0,1$ is obtained in this way, as
each block meets all parallel classes).

All these results can be carried over to an arbitrary triple of pairwise
non-parallel points, using the action of $G$. So we have the following.

\begin{prop}\label{prop:blockstraces} Let $\sigma\ne\id$.
Let $p_1,p_2,p_3\in \cP$ be pairwise non-parallel.  Let $T$ be the intersection
of all blocks through $p_1,p_2,p_3$, and let $C$ be a parallel class not
meeting $T$. Then the following hold.
\begin{enumerate}
\item There is a $g\in G$ such that $T=\PP(F)^g$. \item Each block through
$p_1,p_2,p_3$ meets $C$, and for each $x\in C$ there is a (unique) block
through $p_1,p_2,p_3,x$.
\end{enumerate}
\end{prop}

\begin{cor}\label{cor:blocks}
Let $p_1,p_2,p_3$ be pairwise non-parallel, let $T$ be the intersection of all
blocks through $p_1,p_2,p_3$, and let $x\nparallel p_1,p_2,p_3$. Then the
number of blocks through $p_1,p_2,p_3,x$ is
\begin{itemize}
\item $q$, if  $x\in T$, \item $0$, if $x\notin T$, but $x\parallel x'$ for
some $x'\in T$, \item $1$, otherwise.
\end{itemize}
\end{cor}
Finally, let us point out a particular case:

\begin{cor}\label{cor:4DD}
Let $q$ be even and let $m=2$, i.e., $x^\sigma=x^2$ for all $x\in K$. Then
$\Sigma=(\cP,\cB,\parallel)$ is a $4$-divisible design with parameter
$\lambda_4=1$.
\end{cor}

This result is immediate from Corollary \ref{cor:blocks}, since $F=\GF(2)$
implies now $|T|=|\PP(F)|=3$.

\section{A geometric model}\label{sec:model}
Now we are looking for a geometric point model of the DD $\Sigma$ defined
above, i.e.\ a DD isomorphic to $\Sigma$ whose points are points of a suitable
projective space. We find such a model on the Klein quadric $\cK$ in $\PG(5,K)$
by using H.~Hotje's representation \cite{hot-76}.

\begin{rem}
One could also first find a \emph{line model} of $\Sigma$ in $\PG(3,K)$ (where
the points of $\Sigma$ are certain lines in $3$-space) and then apply the Klein
correspondence. For details on such line models see \cite{blu+h-00b}, in
particular Examples~5.2 and 5.4, and \cite{blunck-03}.
\end{rem}

We embed the ring $R=K(\eps;\sigma)$ in the ring $M =\M(2,K)$ of
$2\times2$-matrices with entries in $K$ via the ring monomorphism
\begin{equation}\label{eq:emb}
    a+b\eps\mapsto \SMat2{a&b\\0&a^\sigma}.
\end{equation}
From now on we identify the ring $R$ with its image under this embedding.
\par
The projective line $\PP(M)$ is defined, mutatis mutandis, according to
(\ref{eq:PR}). The points of $\PP(M)$ are of the form $M(A,B)$, where $(A,B)$
are the first two rows of an invertible $4\times 4$-matrix over $K$, because
(up to notation) $\GL_2(M)$ equals $\GL_4(K)$. Then (\ref{eq:emb}) allows to
identify the point set $\PP(R)$ of $\Sigma$ with a subset of $\PP(M)$.

Now we establish the existence of a bijection $\Phi$ from $\PP(M)$ onto the
Klein quadric $\cal K$. For this we notice that $M$ is a $K$-algebra, with $K$
embedded in $M$ via $x\mapsto \SMat2{x&0\\0&x}$, and that this algebra is
\emph{kinematic}, i.e., each element of $ M$ satisfies a quadratic equation
over $K$. Note that this embedding of $K$ in $M$ is different from the one
obtained from (\ref{eq:emb}), unless $\sigma=\id$. In \cite{hot-76} Hotje
embeds the projective line over an arbitrary kinematic algebra in an
appropriate quadric. For the matrix algebra $M$ this quadric is $\cK$, and the
embedding, which here is a bijection, is the following:
\begin{equation}\label{eq:Hotje}
    \Phi:\PP(M)\to \cK: M(A,B)\mapsto K(\widetilde BA,\det A,\det B),
\end{equation}
where $A,B$ are matrices in $M$, and for  $B=\SMat2{a&b\\c&d}$ we set
$\widetilde B =\SMat2{d&-b\\-c&a}$. The image of $\Phi$ is indeed the Klein
quadric, because $M\times K \times K$ is a $6$-dimensional vector space over
$K$ endowed with the hyperbolic quadratic form $(C,x,y)\mapsto \det C-xy$.

We need the following additional statements:
\begin{prop}\label{rem:hot}
Consider the bijection $\Phi:\PP(M)\to\cK$  given in (\ref{eq:Hotje}), and its
restriction to $\PP(R)$. Then
\begin{enumerate}
\item
The bijection $\Phi$ induces a homomorphism of group actions, mapping
$\GL_2(M)$, acting on $\PP(M)$, to a subgroup of the group of collineations of
$\PG(5,K)$ leaving $\cK$ invariant.
\item
This homomorphism maps the subgroup $\GL_2(R)$, acting on $\cP=\PP(R)$, to  a
subgroup of the group of collineations of $\PG(5,K)$ leaving $\cP^\Phi$
invariant.
\item
Two points of $\PP(R)$ are parallel if, and only if, their $\Phi$-images are
joined by a line contained in $\cK$.
\end{enumerate}
\end{prop}
\begin{proof}
For (a) see \cite[(7.1/2/3)]{hot-76}; (b) follows from (a).

(c): This follows from \cite[(7.5)]{hot-76} and \cite[Prop.~3.2]{blu+h-00b}.
\end{proof}

Writing $K(x_1,x_2,x_3,x_4,x_5,x_6)$ instead of
$K(\SMat2{x_1&x_2\\x_3&x_4},x_5,x_6)$, we obtain by a direct computation that
the mapping $\Phi$ given in (\ref{eq:Hotje}) acts on the points of
$\cP=\PP(R)\subseteq \PP(M)$ as follows:
\begin{equation}\label{eq:map}
R(a+b\eps,1)\mapsto K(a,b,0,a^\sigma,aa^\sigma,1) ;\; R(1,c\eps)\mapsto
K(0,-c,0,0,1,0)
\end{equation}
We shall identify the elements of $\PP(M)$ with their $\Phi$-images. Then, in
particular, we have
\begin{equation}\label{eq:cap}
B_0=\{K(a,0,0,a^\sigma,aa^\sigma,1)\mid a\in K\}\cup\{K(0,0,0,0,1,0)\}.
\end{equation}

In the next lemma we collect some observations, which can be seen directly
using (\ref{eq:map}) and (\ref{eq:cap}).
\begin{lem}\label{lem:cone}
Let $\cP$ and $B_0$ be the point sets in $\PG(5,K)$ from above. Then the
following hold:
\begin{enumerate}
\item
$\cP=\cC\setminus\{S\}$, where $\cC$ is the \emph{cone} with \emph{vertex} $S
=K(0,1,0,0,0,0)$ over $B_0$, i.e.\ the union of all lines joining $S$ with
$B_0$. \item $\cP$ is entirely contained in the hyperplane $\bH$ with equation
$x_3=0$, which is the tangent hyperplane to $\cK$ at $S$.

\item
Two points of $\cP$ are parallel if, and only if, they lie on a
\emph{generator} of $\cC$, i.e. a line through $S$ contained in $\cC$.
\end{enumerate}
\end{lem}

Now we describe the (image of) the base block $B_0$ more closely:

\begin{lem}\label{lem:cap}
Let $B_0$ be as in (\ref{eq:cap}). Then the following hold:
\begin{enumerate}
\item
$B_0$ is a \emph{cap}, i.e.\ a set of points no three of which are collinear.

\item
If $\sigma=\id$, then $B_0$ is  a regular conic; in particular, $B_0$ is
contained in a plane.

\item
If $\sigma\ne\id$, then $B_0$ spans the $3$-space $\bU_0$, given by
$x_2=0=x_3$, complementary to $S$ in $\bH$.
\end{enumerate}
\end{lem}
\begin{proof}
(a): Assume that the line $L$ carries three points of $B_0$. Then $L\subseteq
\cK$. From Proposition \ref{rem:hot}(c) we see that the three points are
pairwise parallel, a contradiction.

(b): Here $B_0=\{K(a,0,0,a,a^2,1)\mid a\in K\}\cup\{K(0,0,0,0,1,0)\}$, which
obviously is a regular conic in the plane spanned by the points
$K(1,0,0,1,0,0)$, $K(0,0,0,0,1,0)$, $K(0,0,0,0,0,1)$ (namely, the intersection
of this plane with the Klein quadric).

(c): In this case, the four vectors
\begin{equation*}
    (0,0,0,0,1,0),\;(0,0,0,0,0,1),\;(1,0,0,1,1,1),
    \mbox{~and~}(a,0,0,a^\sigma,aa^\sigma,1),
\end{equation*}
with $a\in K\setminus F$, are linearly independent, so the point set $B_0$
spans $\bU_0$.
\end{proof}

In case $\sigma=\id$, our geometric model is nothing else than the ``cylinder
model'' of the Miquelian Laguerre plane $\Sigma$: The points are the points of
a cylinder in $3$-space (a quadratic cone minus its vertex), and the blocks are
the regular conics on the cylinder (the intersections with planes complementary
to the vertex). See, e.g., \cite[I.2]{benz-73} for the real case.

We have a closer look at the special case that $\sigma^2=\id$, $\sigma\ne \id$.
Then $q=m^2$ and $K$ is a quadratic extension of $F$. In this case there are
\emph{Baer subspaces}, i.e. spaces coordinatized by $F$, in each projective
space over $K$.

\begin{prop}\label{prop:Baer}
Let $\sigma^2=\id$, $\sigma\ne \id$. Then $B_0$ is an elliptic quadric in the
Baer subspace $\BB\cong \PG(3,m)$ of $\bU_0\cong\PG(3,q)$ defined by the
$F$-subspace
\begin{equation}\label{eq:Baer}
\{(x,0,0,x^\sigma,f_1,f_2)\mid x\in K, f_i\in F\}.
\end{equation}
\end{prop}
\begin{proof}
Obviously, the set in (\ref{eq:Baer}) is a $4$-dimensional subspace of $K^6$,
seen as a vector space over $F$, satisfying the equations $x_2=0=x_3$ and hence
giving rise to a Baer subspace $\BB$ of $\bU_0$. The elements of $B_0$ all lie
in $\BB$. Moreover, by (\ref{eq:cap}), $B_0$ equals the quadric in $\BB$
determined by $N(x)=f_1f_2$, where $N(x)=xx^\sigma$ is the norm of $x\in K$
with respect to the field extension $K:F$ and, in particular, $N$ is a
quadratic form on the vector space $_F K$. Since $B_0$ is a cap by
 \ref{lem:cap} (a), the quadric must be elliptic.
\end{proof}
The quadratic form used in the above is just the restriction to $\BB$ of the
quadratic form describing the Klein quadric. The intersection of the Klein
quadric and $\bm U_0$ is a hyperbolic quadric.

For the rest of this section we consider the case that $\sigma\ne\id$. We try
to describe the geometric model of the DD $\Sigma$ more explicitly. From the
above we know that our base block $B_0$ is a certain cap that spans a $3$-space
$\bU_0$ complementary to $S$ in the tangent hyperplane $\bH\cong\PG(4,K)$ of
$\cK$ at $S$. In the next proposition we describe all blocks. Together with
Lemma \ref{lem:cone} this gives a description of $\Sigma$ in terms of
$\PG(4,K)$.

\begin{prop}\label{prop:blocksGeom}
Let $\sigma\ne \id$. Then the blocks of $\Sigma$ are exactly the intersections
of the cone $\cC$ with the $3$-spaces complementary to $S$ in $\bH$.
\end{prop}
\begin{proof}
We know that $B_0=\cC\cap \bU_0$, with $\bU_0$ complementary to $S$ in $\bH$.
Let $B$ be any block. Then $B=B_0^g$ for some $g\in G=\GL_2(R)\le\GL_2(M)$. By
Proposition \ref{rem:hot}(b), $g$ induces a collineation, say $\widetilde g$,
of $\PG(5,K)$ leaving $\cK$ and $\cP$ invariant. This collineation fixes $S$
(which is the intersection of the lines corresponding to parallel classes) and
its tangent hyperplane $\bH$. So $B$, seen as a set of points in $\bH$, is
$B=B_0^{\widetilde g}=\cC\cap \bU_0^{\widetilde g}$, where $\bU_0^{\widetilde
g}$ is a $3$-space complementary to $S$, as desired. The $3$-space
$\bU_0^{\widetilde g}$ is independent of the choice of $g$, as it is nothing
else than the span of $B$.

So we have  a mapping from the set of blocks to the set of complements of $S$
in $\bH$, which is injective since each complement contains exactly one point
of each generator of $\cC$, i.e.\ of each parallel class of $\cP$, and hence
cannot belong to more than one  block. A simple counting argument shows that
the mapping is also surjective: The number of blocks is $b=|G|/|G_{B_0}|=q^4$
(this can be computed directly, or from (\ref{eq:lambdaA}) using
$\lambda_3=q$), and the number of complements   of $S$ in $\bH$ also is $q^4$,
because they form an affine $4$-space of order $|K|=q$.
\end{proof}

\begin{rem}
The projective model of $\Sigma$ studied in this section is a special case of
the lifted $t$-DDs described in \cite[Cor.~3.3]{blu+h+z-07}. There, the
following geometries are described as $t$-DDs obtained via the lifting process:
Consider an arbitrary finite projective space $\PG(n,q)$ and a set $B_0$ of $k$
points spanning a subspace $\bU_0$ and having the property that any $t$ points
of $B_0$ are independent. Let $\bS$ be a complement of $B_0$. The point set of
the $t$-DD is the cone with basis $B_0$ and vertex $\bS$, minus $\bS$. The
blocks are the intersections of the cone with subspaces complementary to $\bS$,
and two points are parallel if, and only if, together with $\bS$ they span the
same subspace.
\end{rem}
The following is an obvious geometric analogue of Proposition
\ref{prop:blockstraces} and Corollary \ref{cor:blocks}.

\begin{cor}\label{cor:blocksGeom}
Let $p_1,p_2,p_3$ be pairwise non-parallel, let $T$ be the intersection of all
blocks through $p_1,p_2,p_3$, and let $x\nparallel p_1,p_2,p_3$. Then
\begin{enumerate}
\item
$T$ is the intersection of the cone $\cC$ with the plane $\bE$ spanned by
$p_1,p_2,p_3$.
\item
The blocks through $p_1,p_2,p_3,x$ are exactly the intersections of $\cC$ with
$3$-spaces through $\bE$ complementary to $S$. The number of such $3$-spaces is
\begin{itemize}
\item $q$, if  $x\in T$, \item $0$, if $x\notin T$, but $x\parallel x'$ for
some $x'\in T$, \item $1$, otherwise.
\end{itemize}
\end{enumerate}
\end{cor}


Andrea Blunck, Department Mathematik, Universit\"{a}t Hamburg, Bun\-des\-stra{\ss}e 55,
D-20146 Hamburg, Germany,\newline \texttt{andrea.blunck@math.uni-hamburg.de}
\\~\\
Hans Havlicek, Institut f\"{u}r Diskrete Mathematik und Geometrie, Tech\-ni\-sche
Universit\"{a}t Wien, Wiedner Hauptstra{\ss}e 8--10, A-1040 Wien, Austria,\newline
\texttt{havlicek@geometrie.tuwien.ac.at}
\\~\\
Corrado Zanella, Dipartimento di Tecnica e Gestione dei Sistemi Industriali,
Universit\`{a} di Padova, Stradella S. Nicola, 3, I-36100 Vicenza, Italy,\newline
\texttt{corrado.zanella@unipd.it}
\end{document}